\UseRawInputEncoding

\documentclass[a4paper,12pt]{article}

\usepackage{amsmath,amsfonts,amssymb,amsthm,amscd}
\usepackage {amsfonts, amssymb, amscd,amsthm, amsmath, eucal}

\usepackage{amsbsy}
\usepackage[all]{xy}

\theoremstyle{plain} {
  \newtheorem{thm}{Theorem}[section]
  \newtheorem{defn}[thm]{Definition}
  \newtheorem{cor}[thm]{Corollary}
  \newtheorem{lem}[thm]{Lemma}
  \newtheorem{prop}[thm]{Proposition}
  \theoremstyle{definition}
  \newtheorem{rem}[thm]{Remark}
    \newtheorem{constr}[thm]{Construction}
    
  \theoremstyle{plain}
  
  \newtheorem{notation}[thm]{Notation}
  \newtheorem{agr}[thm]{Agreement}
  \newtheorem{setup}[thm]{Setup}
  \newtheorem{sssection}[thm]{}
}

{

}
\renewcommand{\subsubsection}{\sssection\rm}

\newcommand{\bG}{\mathbf G}

\newcommand{\bW}{\mathbb W}

\newcommand{\cE}{\mathcal E}

\newcommand{\cL}{\mathcal L}
\newcommand{\cM}{\mathcal M}
\newcommand{\cO}{\mathcal O}

\newcommand{\cU}{\mathcal U}

\newcommand{\can}{\text{\rm can}}

\newcommand{\pr}{\text{\rm pr}}

\newcommand{\inc}{\text{\rm inc}}

\newcommand{\Aff}{\mathbf {A}}
\newcommand{\Pro}{\mathbf {P}}

\newcommand \xra {\xrightarrow }

\newcommand \hra {\hookrightarrow }

\renewcommand \phi\varphi

\begin{document}

\title{Constant case of the Grothendieck--Serre conjecture in mixed characteristic
}

\author{I.Panin and A.Stavrova\footnote{September the 10-th of 2025}
}

\date{Steklov Mathematical Institute at St.-Petersburg}

\maketitle

\begin{abstract}
Let $D$ be a DVR of mixed characteristic.
Let $\bG$ be a reductive $D$-group scheme.
Then the Grothendieck--Serre conjecture is true for the $D$-group scheme $\bG$ and any
geometrically regular local $D$-algebra $R$. Also we prove a version of
Lindel--Ojanguren--Gabber's geometric presentation lemma in the DVR context.




\end{abstract}

\section{Main Theorem and some others}\label{Introd}
The aim of this preprint is to prove the following
\begin{thm}\label{Main}{\rm
Let $D$ be a D.V.R. of mixed characteristic and $\bG$ be a reductive $D$-group scheme.
Let $R\supseteq D$ be a geometrically regular local
$D$-algebra and $K$ be the fraction field of $R$. Let $E$ be a principal $\bG$-bundle over $Spec\ R$. Suppose
$E$ is trivial over $Spec\ K$. Then it is trivial.
}
\end{thm}

\begin{rem}\label{constant case}
In the other words, for arbitrary DVR ring $D$ of mixed characteristic the Grothendieck--Serre conjecture is true
for all geometrically regular local $D$-algebras $R$ in the "constant" case. That is for all reductive $R$-group
schemes $\bG$ extended from $D$.
\end{rem}

\begin{notation}\label{general}{\rm
We write $V$ for $Spec\ D$, $v\in V$ for the closed point of $V$, $\eta\in V$ for the generic point of $V$.
For each $V$-scheme $S$ write $S_v$ for the closed fibre of $S$ and $S_{\eta}$ for its generic fibre.
}
\end{notation}

\begin{agr}\label{cond_*}{\rm (Condition $(*)$).
Let $M\subseteq \Aff^n_V$ be a closed subset. We will say that $M$ satisfies the condition $(*)$ iff
$codim_{\Aff^n_v}(M_v)\geq 2$ and $codim_{\Aff^n_{\eta}}(M_{\eta})\geq 2$.
}
\end{agr}

\begin{notation}\label{upper_circle}{\rm
We will write $\Aff^{\circ,n}_V$ for any open subscheme of $\Aff^n_V$ of the form $\Aff^n_V-M$,
where $M\subseteq \Aff^n_V$ is closed subjecting the condition $(*)$.
Define open subschemes ${\Pro}^{\circ,n}_V$ of $\Pro^n_V$ similarly.

If $S\subseteq \Aff^n_V$ is a closed subset (say, a divisor), then for each open subset in $\Aff^n_V$ of the form $\Aff^{\circ,n}_V$ write
$S^{\circ}$ for $S\cap \Aff^{\circ,n}_V$ and call $S^{\circ}$ the trace of $S$ in $\Aff^{\circ,n}_V$.

Similarly, for a each closed $T$ in $\Pro^n_V$ and each open subset in $\Pro^n_V$ of the form ${\Pro}^{\circ,n}_V$
write $T^{\circ}$ for $T\cap \Pro^{\circ,n}_V$ and call
$T^{\circ}$ the trace of $T$ in $\Pro^{\circ,n}_V$.
}
\end{notation}



\begin{thm}\label{E_circ}{\rm
Let $\bG$ be a reductive $D$-group scheme, $x\in \Aff^n_V$ a closed point,
$Y\subset \Pro^n_V$ a divisor which does not contain the point $x$.
Let $E$ be a generically trivial $\bG$-bundle over $\Aff^n_V-Y$.
Then there exist a divisor $Y^{ext}$ in $\Pro^n_V$, containing $Y$, an open
${\Pro}^{\circ,n}_V$ in $\Pro^n_V$, a $\bG$-bundle $E^{\circ}$ over ${\Pro}^{\circ,n}_V$ such that \\
(a) $x\in \Aff^n_V-Y^{ext}$; \\
(b) $\Aff^n_V-Y^{ext}\subset {\Pro}^{\circ,n}_V$; \\
(c) the bundles $E^{\circ}|_{\Aff^n_V-Y^{ext}}$ and $E|_{\Aff^n_V-Y^{ext}}$ {\it are isomorphic}.
}
\end{thm}

\begin{proof}[Proof of Theorem \ref{E_circ}]
Let $\bG$, $Y\subset \Pro^n_V$, $x\in \Aff^n_V-Y$ and $E$ be as in Theorem \ref{E_circ}.
The $\bG$-bundle $E$ over $\Aff^n_V-Y$ is generically trivial. By the Nisnevich theorem
there exists a divisor $Z$ in $\Pro^n_V$ which does not contain the closed fibre $\Pro^n_v$
and such that the $E$ is trivial over $(\Aff^n_V-Y)-Z$.
Replacing $Y$ with a divisor $Y^{ext}$ in $\Pro^n_V$, containing $Y$
we may and will suppose that $x\notin Y^{ext}$, $\Pro^{n-1}_V\subset Y^{ext}$ and for
each irreducible component $Z_i$ of the divisor $Z$ the point $x$ is in $Z_i$.
In this case $Y^{ext}\cap Z=M\cup \Gamma_v$, where $M$ subjects the condition (*) as in Agreement \ref{cond_*}
and $\Gamma_v$ is a divisor in $\Pro^n_v$.
Clearly, there exists a $\bG$-bundle $E'$ over
$\Pro^n_V-(Y^{ext}\cap Z)=\Pro^n_V-(M\cup \Gamma_v)$,
which coincides with
$E|_{\Aff^n_V-Y^{ext}}$ over $\Aff^n_V-Y^{ext}=\Pro^n_V-Y^{ext}$ and trivial over $\Pro^n_V-Z$.

Let $S$ be the set of all generic points of the closed subset
$M\cup \Gamma_v$ of $\Pro^n_V$. Each element of $S$ is a codimension two point in $\Pro^n_V$.
Let $\bW\subset \Pro^n$ be the semi-local subscheme whose closed points are exactly the set $S$.
Note that $\bW-S=\bW-(M\cup \Gamma_v)$.

By \cite[Thm. 6.13]{C-T/S} the restriction map
$$H^1_{et}(\bW,\bG)\to H^1_{et}(\bW-S,\bG)$$
is a bijection.
Thus, there is a $\bG$-bundle $E_{\bW}$ over $\bW$ such that
\begin{equation}\label{C-T-S}
E_{\bW}|_{\bW-S}=E'|_{\bW-S}.
\end{equation}
The equality $\bW-S=\bW-(M\cup \Gamma_v)$ and the isomorphism \eqref{C-T-S} show that there exists
a Zariski neiborhood $\cal W$ of the subscheme $\bW$ in $\Pro^n_V$ and a $\bG$-bundle
$E_{\cal W}$ over $\cal W$ such that the $\bG$-bundles
$E_{\cal W}|_{{\cal W}-(M\cup \Gamma_v)}$ and $E'|_{{\cal W}-(M\cup \Gamma_v)}$
are isomorphic.
Thus, there is certain
\begin{equation}\label{E_circ_over_P_circ}
\bG \text{-bundle} \ E^{\circ} \ \text{over} \ [\Pro^n_V-(M\cup \Gamma_v)]\cup {\cal W}
\end{equation}
enjoing conditions
(1) $E^{\circ}|_{\Pro^n_V-(M\cup \Gamma_v)}\cong E'|_{\Pro^n_V-(M\cup \Gamma_v)}$
and
(2) $E^{\circ}|_{\cal W}\cong E_{\cal W}$ as $\bG$-bundles. \\\\
Clearly, the open subset $[\Pro^n_V-(M\cup \Gamma_v)]\cup {\cal W}$ is of the form
${\Pro}^{\circ,n}_V$ in $\Pro^n_V$. It is straightforward to check that
the divisor $Y^{ext}$, containing $Y$, our open subset
${\Pro}^{\circ,n}_V:=[\Pro^n_V-(M\cup \Gamma_v)]\cup {\cal W}$
in $\Pro^n_V$,
the $\bG$-bundle $E^{\circ}$ over this ${\Pro}^{\circ,n}_V$
enjoy the conditions (a),(b) and (c) of the theorem.

\end{proof}

\begin{thm}\label{thm_1}{\rm
Let $\bG$ be a reductive $D$-group scheme, $E^{\circ}$ be a $\bG$-bundle over ${\Pro}^{\circ,n}_V$
for an open ${\Pro}^{\circ,n}_V$ in $\Pro^n_V$. If $E^{\circ}$ is trivial at the generic point,
then $E^{\circ}$ is trivial locally for the Zariski topology.
}
\end{thm}

In the case of finite field $k(v)$ our proof of Theorem \ref{thm_1}
uses the following purely
{\it geometric} result stated right below
and proven in
\cite[Theorem 1.4]{P3}.
In the partial case of $X=\Pro^{n}_V$ this {\it geometric} result is proved in
\cite[Thm. 1.2]{GPan} using a different approach.

\begin{thm}
\label{Major_1}{\rm
Suppose the field $k(v)$ is finite.
Let $X=\Pro^{\circ,n}_V$. That is the closed subset $M=\Pro^n_V-X$ enjoys the condition (*)
as in Agreement \ref{cond_*}. Let $x\in X_v$ be a closed point. Let
$Z\subseteq {\Pro^n_V}$ be a divisor not containing $\Pro^n_v$
such that $x\in Z$.
Write $\mathcal O=\mathcal O_{X,x}$ and $U=Spec(\mathcal O)$.
Then there is a monic polynomial $h\in \cO[t]$,
a commutative diagram
of $V$-schemes (with an irreducible affine $U$-smooth scheme $Y\xra{pr_U \circ \tau} U$)
of the form
\begin{equation}
\label{SquareDiagram2_0}
    \xymatrix{
       (\Aff^1 \times U)_{h}  \ar[d]_{inc} && Y_h:=Y_{\tau^*(h)} \ar[ll]_{\tau_{h}}  \ar[d]^{inc} \ar[rr]^{(p_X)|_{Y_h}} && X-Z  \ar[d]_{inc}   &\\
     (\Aff^1 \times U)  && Y  \ar[ll]_{\tau} \ar[rr]^{p_X} &&  X                                     &\\
    }
\end{equation}
and a $V$-morphism $\delta: U\to Y$, which enjoy the following conditions
\begin{itemize}
\item[\rm{(i)}]
the left hand side square
is an elementary {\bf distinguished} square in the category of affine $U$-smooth schemes in the sense of
\cite[Defn.3.1.3]{MV};
\item[\rm{(ii)}]
the morphism $\delta$ is a section of the morphism $pr_U\circ \tau$ and
$p_X\circ \delta=can: U \to X$, where $can$ is the canonical morphism;
\item[\rm{(iii)}]
$\tau\circ \delta=i_0: U\to \Aff^1 \times U$ is the zero section
of the projection $pr_U: \Aff^1 \times U \to U$;
\end{itemize}
}
\end{thm}

\begin{proof}[Proof of Theorem \ref{thm_1}]
First prove the theorem in the case of the infinite field $k(v)$.
It suffices to prove that for each closed point $x\in {\Pro}^{\circ,n}_V$ the
$\bG$-bundle $E^{\circ}$ is trivial in a Zariski neighborhood of the point $x$.

The $E^{\circ}$ is trivial at the generic point. Thus, by the Nisnevich theorem
there exists a divisor $Z\subseteq {\Pro^n_V}$ not containing $\Pro^n_v$ such that
$E^{\circ}|_{\Pro^{\circ,n}_V-Z}$ is trivial. After a linear change of coordinates
we may suppose that $H_v\nsubseteq Z$ and $x\notin H_v$, where $H=\{T_0=0\}$.

Put $M=\Pro^n_V-\Pro^{\circ,n}_V$. Since $k$ is infinite there exists a $V$-point $\mathbb V\subset H$
such that \\
(1) $\mathbb V\subset \Pro^{\circ,n}_V\cap H$ and $\mathbb V \cap Z=\emptyset$; \\
(2) for the $V$-linear projection $q_{\mathbb V}: \Pro^n_V-\mathbb V \to \Pro^{n-1}_V$
one has $q^{-1}_{\mathbb V}(q_{\mathbb V}(x))\cap M=\emptyset$.

After another linear change of coordinates we may suppose that
$\mathbb V=[0:...:1]_V$. Write shortly $q$ for $q_{\mathbb V}$.
Since $\mathbb V=[0:...:1]_V$ one has $q(x)\in \Aff^{n-1}_V$.
Put $S=Spec ({\cal O}_{\Aff^{n-1}_V,q(x)})$. The Blowing up $\Pro^{\circ,n}_V$ at $\mathbb V$
gives rise to a commutative diagram of the form
\begin{equation}
\label{SquareDiagram}
    \xymatrix{
    \Aff^1_S \ar[drr]_{q}\ar[rr]^{j}&&
\Pro^1_S \ar[d]_{\bar q}  && \infty\times S \ar[ll]_{i}\ar[lld]^{id} &\\
     && S  &\\    },
\end{equation}
where $q$ and $\bar q$ are the projections to $S$. Furthermore there is
the blowing down morphism
$\sigma: (\Pro^{\circ,n}_V)^{\wedge}_{\mathbb V} \to \Pro^{\circ,n}_V$.
{\it Write still $\sigma$ for its restriction to} $\Pro^1_S$.

Then $\sigma \circ j: \Aff^1_S\to \Pro^{\circ,n}_V$ is the inclusin
$in: \Aff^1_S\to \Pro^{\circ,n}_V$. By the definition of $S$ the point $q(x)$ is its closed point.
Thus, $x$ is in $\Aff^1_S$.

Consider the $\bG$-bundle $\sigma^*(E^{\circ})$ over $\Pro^1_S$.
Since $Z\cap \mathbb V=\emptyset$ and $E^{\circ}$ is trivial away of $Z$
it follows that
the $\bG$-bundle $\sigma^*(E^{\circ})|_{\infty\times S}$ is {\it trivial}.
By \cite[Corollary 1.14]{PSt2} the $\bG$-bundle $\sigma^*(E^{\circ})$
over $\Pro^1_S$
is Zariski
locally trivial.

Since
$\sigma \circ j=in: \Aff^1_S\hra \Pro^{\circ,n}_V$ it follows that
$j^*(\sigma^*(E^{\circ}))=E^{\circ}|_{\Aff^1_S}$. Thus, $E^{\circ}|_{\Aff^1_S}$
is Zariski locally trivial. So, $E^{\circ}$ is trivial in a Zariski
neighborhood of the point $x$.\\\\
Now prove the finite $k(v)$ field case of Theorem \ref{thm_1}.
It suffices to prove that for each closed point $x\in {\Pro}^{\circ,n}_v$ the
$\bG$-bundle $E^{\circ}$ is trivial in a Zariski neighborhood of the point $x$.
The $\bG$-bundle $E^{\circ}$ over $\Pro^{\circ,n}_V$ is trivial at the generic point.
Thus, by the Nisnevich theorem
there exists a divisor $Z\subseteq \Pro^n_V$ not containing $\Pro^n_v$ such that
$E^{\circ}|_{\Pro^{\circ,n}_V-Z}$ is trivial.
If $x\notin Z_v$, then there is nothing to prove.

So, we suppose below in the proof that $x\in Z_v$. Put $X=\Pro^{\circ,n}_V$.
The divisor $Z$ in $\Pro^n_V$ does not contain $\Pro^n_v$ and $x\in Z_v$.
By Theorem \ref{Major_1} there is a monic polynomial $h\in \cO[t]$ and
a commutative diagram of $V$-schemes of the form \eqref{SquareDiagram2_0}
enjoying conditions (i)--(iii).

The $\bG$-bundle $E^{\circ}$ over $X$ is such that its restriction to $X-Z$ is trivial.
Write ${\cal E}$ for the $\bG$-bundle $p^*_X(E^{\circ})$ over $Y$. Then
its restriction to the open subset $Y_h:=Y_{\tau^*(h)}$ of $Y$ is trivial.
The left hand side square in the diagram \eqref{SquareDiagram2_0}
is an elementary {\bf distinguished} square in the category of affine $U$-smooth schemes in the sense of
\cite[Defn.3.1.3]{MV}. Thus, there is a $\bG$-bundle $E_t$ over $\Aff^1\times U$
such that $E_t|_{(\Aff^1\times U)_h}$ is trivial and $\tau^*(E_t)\cong {\cal E}$.

The polynomial $h\in \cO[t]$ is monic. Thus the closed subset $\{h=0\}$ of $\Aff^1\times U$
is closed in $\Pro^1\times U$. So, its complement is open in $\Pro^1\times U$.
Clearly, there is a $\bG$-bundle $\bar E_t$ over $\Pro^1\times U$ such that
$\bar E_t|_{\Aff^1\times U}=E_t$ and $\bar E_t|_{\Pro^1\times U-\{h=0\}}$ is trivial.
Particularly, the $\bG$-bundle $\bar E_t|_{\infty \times U}$ is trivial.
By \cite[Thm. 1.12]{PSt2} the $\bG$-bundle $\bar E_t|_{0 \times U}$ is trivial.

Thus,
the $\bG$-bundle $E_t|_{0 \times U}$ is trivial.
The properties (ii) and (iii) of the diagram \eqref{SquareDiagram2_0}
show that $E_t|_{0 \times U}\cong E|_U$ as $\bG$-bundles.
Thus, the $\bG$-bundle $E|_U$ is trivial.

The proof of the finite $k(v)$ field case of Theorem \ref{thm_1} is completed.
Theorem \ref{thm_1} is proved.

\end{proof}

\section{Proof of Theorem \ref{Main}}\label{reduction}
We will use in the proof the following {\it terminology}. Namely,
by a {\it localization} of an affine scheme $S=(Spec\ A,\cO_{Spec\ A})$ we mean a subscheme $\cal S$ of $S$ of the form
$(Spec\ A_M,\cO_{Spec\ A_M})$, where $M\subset A$ is a multiplicative system. Clearly, for each point $s\in \cal S$ the local rings
$\cO_{S,s}$ and $\cO_{{\cal S},s}$ coincide.

We also will use in the proof the notion of
{\it an elementary distinguished square}
in the sense of
\cite[Defn.3.1.3]{MV}.

\begin{proof}[Proof of Theorem \ref{Main}]
Let $\pi \in D$ be a generator of the maximal ideal of $D$.
Using the D.Popescu theorem \cite{Po} and the limit arguments
we may suppose that the $D$-algebra $R$ is $D$-smooth
{\it integral domain}.
Furthermore we may and will suppose that
the $\bG$-bundle $E$ is defined over $R$ and is {\it trivial} over the fraction field of $R$.
It is sufficient to prove that the $\bG$-bundle $E$ is locally trivial for the Zariski topology on $\text{Spec} R$.
Put $X=\text{Spec}\ R$.
There are two cases: either the closed fibre $X_v$ of $X$ is empty or $X_v$ is non-empty.

If $X_v=\emptyset$, then our $R$ contains the fraction field $L$ of $D$. The field $L$ is infinite.
Thus, the $\bG$-bundle $E$ is locally trivial for the Zariski topology on $X$
by \cite[Theorem 1.1]{FP}.

If $X_v\neq \emptyset$, then by \cite[Theorem 1]{NG}
there is a closed subset $Z$ of {\it pure codimension one} in $X$
such that
$Z$ does not contain any component of $X_v$
and the $\bG$-bundle $E|_{X-Z}$ is trivial.
It is sufficient to check for each closed point $x\in X$ that
the $\bG$-bundle $E$ is trivial in a Zariski neighborhood of the point $x$.
If $x\notin Z$, then there is nothing to prove. Take now $x\in Z$.
If $x\notin Z_v$, then $x\in Z_{\eta}\subset X_{\eta}$.
Note that $X_{\eta}$ is open in $X$.
Thus, using again
\cite[Theorem 1.1]{FP}
we see that the $\bG$-bundle $E$ is trivial in a Zariski neighborhood of the point $x$.
It remains to prove that for a closed point $x\in Z_v$
the $\bG$-bundle $E$ is trivial over $Spec\ \cO_{X,x}$.
This is done in the rest of the proof.


Regard the closed subset $Z$ as a Cartier divisor in $X$.
By Theorem \ref{geometric_form_of_O-G-L}
there are a
localization $X'$ of the scheme $X$ containing the point $x$,
a closed point $y\in \Aff^n_v$, the local subscheme
${\cal W}:=Spec\ \cO_{\Aff^n_V,y}$ of the $\Aff^n_V$,
an element $g\in m_y\subset \cO_{\Aff^n_V,y}$
and an elementary distinguished square of the form
\begin{equation}
\label{Nisn_square}
    \xymatrix{
X'_{\tau^*(g)}  \ar[d]^{\tau_g}  \ar[rr]^{In} && X' \ar[d]^{\tau}  &\\
{\cal W}_g  \ar[rr]^{in} &&   \cal W   &,\\    }
\end{equation}
in the category $Sm/{\cal W}$ of $\cal W$-smooth schemes enjoying the following properties: \\
(1) $\tau(x)=y$; \\
(2) $\{\tau^*(g)=0\}=X'\cap Z$ as the Cartier divisors of the $X'$.\\
Put $Z'=X'\cap Z$ and recall that
the $\bG$-bundle $E|_{X-Z}$ is trivial. Thus,
the $\bG$-bundle $E|_{X'-Z'}$ is trivial as well.
Since the square \eqref{Nisn_square}
is an elementary distinguished square, it follows that there is
a $\bG$-bundle $\cE$ over $\cal W$ such that
$\tau^*(\cE)=E|_{X'}$ and the bundle $\cE|_{{\cal W}_g}$ is trivial
(see \cite[Proposition 2.6]{C-TO} for details).

Futhermore, we may and will suppose that the $\bG$-bundle $\cE$ is defined over a
principal open subset $\Aff^n_V-Y$ ($Y$ is a divisor not containing the point $y$).
Clearly, the $\bG$-bundle $\cE$ over $\Aff^n_V-Y$ is generically trivial.
Now by Theorem \ref{E_circ}
there exist a divisor $Y^{ext}$ in $\Pro^n_V$, containing $Y$, an open
${\Pro}^{\circ,n}_V$ in $\Pro^n_V$, a $\bG$-bundle $E^{\circ}$ over ${\Pro}^{\circ,n}_V$ such that \\
(a) $y\in \Aff^n_V-Y^{ext}$; \\
(b) $\Aff^n_V-Y^{ext}\subset {\Pro}^{\circ,n}_V$; \\
(c) the bundles $E^{\circ}|_{\Aff^n_V-Y^{ext}}$ and $\cE|_{\Aff^n_V-Y^{ext}}$ {\it are isomorphic}.\\
Since the $\bG$-bundle $\cE$ is generically trivial the property (c) yields that the $\bG$-bundle $E^{\circ}$ is generically trivial as well.
By Theorem \ref{thm_1} the $\bG$-bundle $E^{\circ}$ is Zariski locally trivial.
Thus, the $\bG$-bundles $\cE$ is trivial in a neighborhood of the point $y$.
Since $\tau^*(\cE)=E|_{X'}$ and $\tau(x)=y$
it follows that the $\bG$-bundles $E$ is trivial in a neighborhood of the point $x$.
That is the $\bG$-bundles $E|_{Spec\ \cO_{X,x}}$ is trivial.
The theorem is proved.

\end{proof}

\section{Gabber--Lindel--Ojanguren's lemma over a DVR}

\begin{notation}\label{general}{\rm
Let $D$ be a D.V.R. of mixed characteristic $(p,0)$.
Let $m$ be the maximal ideal of $D$. Put $k(v)=A/m$. It is a characteristic $p$ field.
Let $K$ be the fraction field of $D$. It is a characteristic zero field.

In this notes by a $D$-algebra we mean a $D$-algebra which is a Noetherian as a ring.
For a $D$-algebra $A$ write $\bar A$ for $A/mA$. Recall that a
{\it $D$-smooth} algebra $A$
is necessary a finitely generated $D$-algebra.
}
\end{notation}

The main result of these section is this.

\begin{thm}\label{O-G-L-DVR}{\rm
Let $A$ be a smooth $D$-algebra with $\bar A\neq (0)$ and which is an integral domain.
Let $m_x$ be a maximal ideal in $A$ with $\bar m_x\neq \bar A$ and let $f\in m_x$.
Suppose $\bar f \in \bar A$ is not a {\it zero divisor}.

Then there is a multiplicative subsystem $S\subset A-m_x$,
a closed point $y\in \Aff^n_v$, an element $g\in m_y$
and a $D$-algebra inclusion
$inc: \cO_{\Aff^n_V,y}\hra A_S$ which enjoys the following properties: \\
(0) $m_y=(m_x)_S \cap \cO_{\Aff^n_V,y}$; \\
(1) the inclusion $inc$ is \'{e}tale; \\
(2) $(g)=(f)\cap \cO_{\Aff^n_V,y}$ in $\cO_{\Aff^n_V,y}$; \\
(3) the induced homomorphism
$\cO_{\Aff^n_V,y}/g\cO_{\Aff^n_V,y}\to A_S/gA_S$
is an isomorphism.
}
\end{thm}
{\it To prove this result it is convenient to restate it in geometric terms as Theorem \ref{geometric_form_of_O-G-L}}.
\begin{setup}\label{setup:basic}{\rm
We write $V$ for $Spec\ D$, $v\in V$ for the closed point of $V$, $\eta\in V$ for the generic point of $V$.
For a $V$-scheme $S$ write $S_v$ for its closed fibre of $S$ and $S_{\eta}$ for its generic fibre.

Let $X$ be an irreducible affine $V$-smooth scheme of pure relative dimension $n\geq 1$
(particularly its closed fibre $X_v$ is non-empty). Let $x\in X_v$ be a closed point.

Let $Z\subset X$ be a closed subset of {\it pure codimension one} in $X$ regarded as a
Cartier divisor of $X$. It is supposed that $x\in Z_v$ and $Z$ does not contain any component of $X_v$.
Let $I\subset \Gamma(X,\cO_X)$ be ideal defining $Z$ in $X$.
{\it This is the set up for these notes}.
}
\end{setup}

\begin{defn}\label{principal_open}{\rm
Let $A$ be a finitely generated reduced $D$-algebra which is an integral domain.
Let $Y$ be the affine scheme $(Spec\ A,\cO_{Spec\ A})$.
By a {\it principal} open subscheme of $Y$ we mean in this notes its open
subscheme of the form $(Y_f,\cO_Y|_{Y_f})$, where $0\neq f\in A$.
Recall that principal open subsets $Y_f$ of $Y$ form a basis of the Zariski topology on $Y$.

By a {\it localization} of an affine scheme $Y=(Spec\ A,\cO_{Spec\ A})$ we mean a subscheme $\cal Y$ of $Y$ of the form
$(Spec\ A_S,\cO_{Spec\ A_S})$, where $S\subset A$ is a multiplicative system.
}
\end{defn}
To state the following result we use the notion of
{\it an elementary distinguished square}
in the sense of
\cite[Defn.3.1.3]{MV}.

\begin{thm}\label{geometric_form_of_O-G-L}{\rm
Let $X/V$, an integer $n\geq 1$, a closed subset $Z\subset X$, a closed point $x\in Z_v$
be as in Setup \ref{setup:basic}.
Then there are a
localization $\cal X$ of the scheme $X$ containing the point $x$,
a closed point $y\in \Aff^n_v$, the local subscheme
${\cal W}:=Spec\ \cO_{\Aff^n_V,y}$ of the $\Aff^n_V$,
an element $g\in m_y\subset \cO_{\Aff^n_V,y}$
and an elementary distinguished square of the form
\begin{equation}
\label{Nisn_Sq_2}
    \xymatrix{
{\cal X}_{\tau^*(g)}  \ar[d]^{\tau_g}  \ar[rr]^{In} && \cal X \ar[d]^{\tau}  &\\
{\cal W}_g  \ar[rr]^{in} &&   \cal W   &,\\    }
\end{equation}
in the category $Sm/{\cal W}$ of $\cal W$-smooth schemes enjoying the following properties: \\
(1) $\tau(x)=y$; \\
(2) $\{\tau^*(g)=0\}={\cal X}\cap Z$ as the Cartier divisors of the $\cal X$.
}
\end{thm}

\section{Proof of Theorem \ref{geometric_form_of_O-G-L}}
A comment concerning the terminology used in this section:
if we replace (shrink) the $X$ with its open subset $X'$, then
we replace its closed subset $Z$ with $Z'=Z\cap X'$.
We say in this case that we replace (shrink) the $Z$ {\it accordingly}.

The proof of Theorem \ref{geometric_form_of_O-G-L} consists of five steps.
\begin{proof}[Proof of Theorem \ref{geometric_form_of_O-G-L}]

{\it Step 1}. The following Lemma is true since $X$ is $V$-smooth.
\begin{lem}\label{Just_Etale_at_x}{\rm
Let $X/V$, an integer $n\geq 1$, a closed subset $Z\subset X$, a closed point $x\in Z_v$
be as in Setup \ref{setup:basic}.
Shrinking $X$ such that $x$ {\it belongs to the shrunk} $X$
and replacing $Z$ accordingly
we may and will suppose that
$x\in Z_v\subset X_v$ and
there is
an open embedding $i: X\hra \bar X$, a finite surjective $V$-morphism
$\bar \rho: \bar X\to \Pro^n_V$ enjoying the following properties \\
(a) the $V$-scheme $\bar X$ is projective, irreducible, reduced and normal; \\
(b) $\bar \rho(X)\subset \Aff^n_V$; \\
(c) $\bar \rho|_X: X\to \Aff^n_V$ is \'{e}tale.
}
\end{lem}

The following well-known lemma was communicate via e-mail to us by K. \v{C}esnavi\v{c}ius.
\begin{lem}\label{reg_locus}{\rm
Let$À$ be an integral domain finitely generated over our D.V.R $D$.
Suppose the fraction field $\text{Frac}(A)$ is of characteristic zero.
Suppose $A$ is normal. Put $Y = \text{Spec} A$.
Let $Y_{reg}$ be the subset of regular points of $Y$. Then \\
(1) $Y_{reg}$ is an open subset in $Y$;\\
(2) $Y - Y_{reg}$ has codimension at least 2 in $Y$.
}
\end{lem}

\begin{proof}
A mixed characteristic DVR is excellent. Thus any ring / scheme that is of finite type over it is also excellent
(https://stacks.math.columbia.edu/tag/07QW).
It is then a general fact that for an excellent scheme the locus of regular points is open
(EGA IV, 7.8.3 (iv)).
This proves the assetion (1).

Now to prove the assertion (2) it suffices to check that $Y_{reg}$ contains every point $y$ of $Y$ of height $\le 1$.
The local domain $\cO_{Y, y}$ is normal of dimension $\le 1$, so it is either a field or a DVR.
In particular, $\cO_{Y, y}$ is regular, so $y\in Y_{reg}$, as desired.
\end{proof}

\begin{cor}\label{reg_locus_2}{\rm
Let $\bar X_{reg}$ be the subset of regular points of $\bar X$. Then \\
(1reg) $\bar X_{reg}$ is an open subset in $\bar X$, \\
(2reg) $\bar X-\bar X_{reg}$ has codimension at least 2 in $\bar X$.
}
\end{cor}
The following lemma is a consequence of Corollary \ref{reg_locus_2}.


\begin{lem}\label{X_reg}{\rm
Let $X$, $x\in X_v$, $\bar X$, $\rho: \bar X\to \Pro^n_V$ be as in Lemma \ref{Just_Etale_at_x}.
Replacing $X$ with its affine open subscheme $X'$ one can find an affine open subscheme $\tilde X$ in $\bar X$ such that \\
(0) $X'\subset \tilde X\subset \bar X$ (both inclusions are open); \\
(1) $x\in X'$ and $X'$ is $V$-smooth.\\
(2) the scheme $\tilde X$ is regular;\\
(3) $\tilde X_v$ is the disjoint union of its irreducible components $\tilde X_{v,j}$ ($j\in J$);\\
(4) $X'=\tilde X-(\sqcup_{j\in J''} \tilde X_{v,j})$, where $J''=\{j\in J| X'\cap \tilde X_{v,j}=\emptyset \}$;\\
Moreover, \\
(5) $\text{dim} (\bar X_v - \tilde X_v)=n-1$; \\
(6) for each $j\in J$ one has $\text{dim}\tilde X_{v,j}=n$;\\
(7) put $J'=J-J''$, then for each $j\in J'$ the $\tilde X_{v,j}$ is $k(v)$-smooth and $\tilde X_{v,j}\subset X'$.
}
\end{lem}

\begin{proof}
Let $\theta\in D$ be a local parameter. Since $\bar X$ is irreducible it follows that $\theta$
is not a zero-divisor in each local ring $\cO_{\bar X,y}$ for $y\in \bar X_v$.
Thus, the closed subscheme $\bar X_v$ of $\bar X$ is equidimensional
of dimension $n$. Let $\bar X_{v,i}$ ($i\in I$) be all its irreducible components. Then for each $i\in I$ one has $\text{dim}\bar X_{v,i}=n$.
And for each $i\neq i'$ in $I$ one has $\text{dim}(\bar X_{v,i}\cap \bar X_{v,i'})< n$.

Let $\bar X_{reg}$ be the regular locus of $\bar X$. By Corollary \ref{reg_locus_2} $\bar X_{reg}$ is open in $\bar X$.
Thus, $\text{Sing}(\bar X):=\bar X-\bar X_{reg}$ is closed in $\bar X$. By the same Corollary it is of codimension at least $2$ in $\bar X$. So,
$\text{Sing}(\bar X)$ does not contain any of $\bar X_{v,i}$. Since $X$ is $V$-smooth it follows that $X\subset \bar X_{reg}$.
As a consequence, $x\notin \text{Sing}(\bar X)$.

Now put $T_{\eta}:=\bar X_{\eta}-X_{\eta}$. It is closed in $\bar X_{\eta}$. Since $\bar X_{\eta}$ is irreducible of dimension $n$,
it follows that the dimension of $T_{\eta}$ is at most $n-1$.
Let $\bar T_{\eta}$ be the closure of $T_{\eta}$ in $\bar X$.
Then the dimension of $\bar T_{\eta}$ is at most $n$.
Let $T_s$ be an irreducible component of $\bar T_{\eta}$. Then $(T_s)_v$ is its proper closed subset.
Thus, its dimension is at most $n-1$. It follows that as $T_s$, so $\bar T_{\eta}$ does not contain
any of $\bar X_{v,i}$. Since $\bar X_{\eta}-X_{\eta}\subset \bar X-X$ and $\bar X-X$ is closed in $\bar X$
it follows that $\bar T_{\eta}\subset (\bar X-X)$. Thus $x\notin \bar T_{\eta}$.

Put $Y=\text{Sing}(\bar X)\cup \bar T_{\eta}\cup \bigcup_{(i,i')|i\neq i'}(\bar X_{v,i}\cap \bar X_{v,i'})$.
For each $i\in I$ choose a closed point $x_i\in \bar X_{v,i}$, which is not in $Y_v$ and which is different of the point $x$.
It is available to do this since as $\text{Sing}(\bar X)$, so $\bar T_{\eta}$ do not contain any of $\bar X_{v,i}$.
We checked that $x\notin (\text{Sing}(\bar X)\cup \bar T_{\eta})$.
Since the scheme $\bar X_v$ is smooth at the point $x$, it follows that $x$ is not in any of $\bar X_{v,i}\cap \bar X_{v,i'}$ for $i\neq i'$.

Put $\cL:=\bar \rho^*(\cO_{\Pro^n_V}(1))$.
Since $\rho$ is finite the $\cL$
is ample.
Put $Y^{ext}:=Y\sqcup \{x_i\}_{i\in I}\sqcup \{x\}$.
By the Serre vanishing theorem there exist an integer $m>>0$ such that
$H^0(\bar X, \cL^{\otimes m})$ surjects on $H^0(Y^{ext},\cL^{\otimes m}|_{Y^{ext}})$.

Thus, there exists a global section $s$ of $\cL^{\otimes m}$ which identically vanishes on
$Y$ and $s(x)\neq 0$ and for each $i\in I$ $s(x_i)\neq 0$.

Let $[T_0:T_1:...:T_n]$ be homogeneous coordinates on the $\Pro^n_V$
such that the $\Aff^n_V=(\Pro^n_V)_{T_0}$. By Lemma \ref{Just_Etale_at_x} one has
$X\subset \bar X_{s_{inf}}$, where $s_{\inf}:=\bar \rho^*(T_0)$.
Thus, $X_s\subset \bar X_{s\cdot s_{\inf}}$. Put
$$X'=X_s \ \ \ \text{and} \ \ \ \tilde X=\bar X_{s\cdot s_{\inf}}.$$
Clearly, $X'\subset \tilde X\subset \bar X$ are open subschemes.

Prove that as $X'$, so $\tilde X$ is an affine scheme.
The scheme $X_{aff}:=\bar X_{s_{inf}}$ is affine since it is finite over the $\Aff^n_V$.
Thus, its open subscheme $\bar X_{s\cdot s_{inf}}$ is also affine as the principal open subscheme of $X_{aff}$
corresponding to the regular function
$f:=Res_{X_{aff}}(s)/(Res_{X_{aff}}(s^{\otimes m}_{\inf})$ on $X_{aff}$.
Write $f$ for the restriction $f|_X$. {\it Then} $X':=X_f$. Since $X$ is affine it follows that $X'$ is also an affine scheme.

We left to the reader to check that for
for these $X'$ and $\tilde X$
the assertions (1) to (7) are true.

\end{proof}

\begin{agr}\label{Z_reg_and X_circ}{\rm
Put $Z'=Z\cap X'$. Then one has $x\in Z'_v$.
To simplify notation write below in the preprint $X$ for $X'$, $\bar \rho|_X$ for $\bar \rho|_{X'}$ and $Z$ for $Z'$.
Note that $\bar \rho|_X$ remains \'{e}tale and $x$ remains in $Z_v$.
}
\end{agr}

\begin{notation}\label{X_Z_bar Z_bar X_X_v}{\rm
Let $\bar Z$ be the closure of $Z$ in $\bar X$.
We will work below with {\it the chosen $\tilde X$ and with $\tilde Z:=\bar Z\cap \tilde X$}. Clearly,
{\it the closure} of the $\tilde Z$ in $\bar X$ coincides with $\bar Z$.
Note that $\tilde Z$ is closed in $\tilde X$ and it does not contain any irreducible
component $\tilde X_{v,j}$ of the $\tilde X_v$.
}
\end{notation}

\begin{rem}\label{X_reg_v}{\rm
By Lemma \ref{X_reg} the $k(v)$-scheme $\tilde X_v$ is equidimensional of dimention $n$ and
for each $j\in J'$ its irreducible component $\tilde X_{v,j}$ is $k(v)$-{\it smooth}.
However for $j\in J-J'$ the $k(v)$-scheme $\tilde X_{v,j}$ can be {\it non-reduced}.
}
\end{rem}

%
%
{\it Step 2}.
\begin{notation}\label{closed_fibre}{\rm
Let $Y$ be an irreducible affine $V$-smooth scheme of pure relative dimension $n\geq 1$
(particularly its closed fibre $Y_v$ is non-empty).
Let $y\in Y_v$ be a closed point. We write $m_{Y_v,y}$ for the maximal
{\it ideal in}
$\Gamma(Y_v,\cO_{Y_v})$ corresponding to the point $y\in Y_v$.

We write $y^{(2)}$ for the closed subscheme $Spec\ \cO_{Y_v,y}/m^2_{Y_v,y}$ of the scheme $Y_v$.
}
\end{notation}

\begin{lem}\label{two_fields}{\rm
Put $k=k(v)$. Let $T$ be an affine $k$-scheme of finite type of pure dimension $d$.
Suppose $T$ is $k$-smooth at a closed point $t\in T$. Then there is a closed point $y\in \Aff^d_k$ such that
$k(t)$ and $k(y)$ are isomorphic as the $k$-algebras.
}
\end{lem}
\begin{proof}
In the case of an infinite field $k$ this is a consequence of Ojanguren's Lemma
\cite[Ch.VIII,Corollary 3.5.2]{Knu}.
In the case of a finite field $k$ this is a consequence of
Lindel's Lemma
\cite[Proposition 2]{L}.
\end{proof}

\begin{lem}\label{two_points_and_Cohen}{\rm
In the hypotheses of Lemma \ref{two_fields} the $k$-algebras
$k[t^{(2)}]=k[T]/m^2_t$ and $k[y^{(2)}]=k[\Aff^d_k]/m^2_y$ are
isomorphic.
}
\end{lem}

\begin{proof}
By Lemma \ref{two_fields} there is a closed point $y\in \Aff^d_k$ such that
$k(t)$ and $k(y)$ are isomorphic as the $k$-algebras.
Since our $T$ is $k$-smooth at a closed point $t\in T$ it follows by the Cohen theorem that
the completed $k$-algebra $\widehat{k[T]}_{m_t}$ is isomorphic to the completed $k$-algebra
$\widehat{k[\Aff^d_k]}_{m_y}$.
Thus, the $k$-algebras
$k[t^{(2)}]$ and $k[y^{(2)}]$ are isomorphic.
\end{proof}

{\it Step 3.}
First, let $X/V$, an integer $n\geq 1$, the closed subset $Z\subset X$, the closed point $x\in Z_v$
be as in Setup \ref{setup:basic}.
Let $\bar X$, $\bar \rho: \bar X\to \Pro^n_V$
be as in the conclusion of Lemma \ref{Just_Etale_at_x}.
Let $X'\subset \tilde X\subset \bar X$ be as in Lemma \ref{X_reg}.
Finally, take into account Agreement \ref{Z_reg_and X_circ} and
use Notation \ref{X_Z_bar Z_bar X_X_v}.
Then $x\in X$, $X\subseteq \tilde X$ and $X$ is $V$-smooth.
Particularly, the $V$-scheme $\tilde X$ is $V$-smooth at the point $x$.
Also, the morphism $\bar \rho|_X$ is \'{e}tale and $x$ is in $Z_v$.
We also have the $\bar Z$ in $\bar X$ with $Z=X\cap \bar Z$ and $\tilde Z=\tilde X\cap \bar X$.


The aim of this Step is to prove the following
\begin{prop}\label{bar pi}{\rm
Let $\cL_{new}:=\cL^{\otimes (m+1)}$ and
$s_{new}=s\cdot s_{\inf}\in \Gamma(\bar X,\cL_{new})$, where $\cL^{\otimes (m+1)})$ and
$s\cdot s_{\inf}\in \Gamma(\bar X, \cL^{\otimes (m+1)})$ be as in the proof of
Lemma \ref{X_reg}.
Let
$X_{\infty}\subset \bar X$ be the Cartier divisor corresponding to the section $s_{new}$
(particularly, $\tilde X=\bar X-X_{\infty}$).
Then
there is a finite surjective $V$-morphism
\begin{equation}
\label{morphism:bar pi}
\bar \pi: \bar X \to \Pro^{n,w}_V,
\end{equation}
where $\Pro^{n,w}_V$ is the weighted projective space with homogeneous coordinates
$[t_0:t_1:...:t_n]$ of certain weights
$1,e_1,...,e_n$ (with $e_i$ divides $e_{i+1}$).
Moreover $(\Pro^{n,w}_V)_{t_0}=\Aff^n_V$ and
the morphism $\bar \pi$ enjoy the following properties: \\
(0) $\Tilde X=\bar \pi^{-1}((\Pro^{n,w}_V)_{t_0})$; (write below in this section $\pi$ for $\bar \pi|_{\Tilde X}: \Tilde X\to \Aff^n_V$); \\
(1) the morphism $\pi: \Tilde X\to \Aff^n_V$ is a finite and flat; \\
(2) the $\pi_v|_{x^{(2)}}: x^{(2)}\to \pi(x)^{(2)}$ is a scheme isomorphism (here $\pi_v=\pi|_{X_v}: \Tilde X_v \to \Aff^n_v$);\\
(3) $\bar \pi^{-1}(\bar \pi(x))\cap \bar Z=\{x\}$ and $\pi^{-1}(\pi(x))\cap \Tilde Z=\{x\}$. \\
As a consequence the following is true:\\
(4) $\pi$ is \'{e}tale at $x\in \Tilde X$,\\
(5) for a principal divisor $E\subset \Tilde X$ and the principal open subscheme
$X^{\circ}:=\Tilde X-E$ the $\pi|_{X^{\circ}}: X^{\circ}\to \Aff^n_V$ is \'{e}tale and $x\notin E$,\\
(6) $\pi^*: k(\pi(x))\to k(x)$ is a field isomorphism.
}
\end{prop}

\begin{proof}[Proof of Proposition \ref{bar pi}]
The proof is rather long. In particular it requires to state and
to prove Lemmas \ref{n-1} and \ref{n}. In turn, to prove these Lemmas we
state and prove a Claim right below.

Note that $\cL_{new}$ is an ample locally free $\cO_{\bar X}$-module. Thus, there is an integer $M>>0$
such that the locally free $\cO_{\bar X}$-module $\cM:=\cL_{new}^{\otimes M}$ is very ample.
So, there is a closed embedding of the $V$-schemes
$in: \bar X \hra \Pro^N_V$
such that
$\cM=inc^*(\cO_{\Pro^N_V}(1))$.
Let $in_v: \bar X_v \hra \Pro^N_v$ be the closed $v$-embedding of the closed fibres.

For each
{\it finite subscheme}
$Y$ of $\bar X_v$ consider the commutative square
(point out that in the left bottom corner we insert $\bar X$ rather than $\bar X_v$)
\begin{equation}
\label{Poonen}
    \xymatrix{
H^0(\Pro^N_V,\cO_{\Pro^N_V}(d))  \ar[d]^{Res_{\bar X}}  \ar[rr]^{Res_{\text{P}_v}} && H^0(\Pro^N_v,\cO_{\Pro^N_v}(d)) \ar[d]^{res_Y}  &\\
H^0(\bar X,\cM^{\otimes d})  \ar[rr]^{Res_Y} &&   H^0(Y,\cO_Y(d))   &,\\    }
\end{equation}
The map $res_Y$ is surjective if $d\geq \text{dim}_{k(v)}H^0(Y,\cO_Y)-1$ by \cite[Lemma 2.1]{Poo}.
That Lemma is stated and proved for an arbitrary base field.
Thus, the cited Lemma yields the following

{\it Claim.} The homomorphism $Res_Y$
is surjective for each $d\geq \text{dim}_{k(v)}H^0(Y,\cO_Y)-1$.\\\\
Consider now the finite subscheme $x^{(2)}= Spec\ \cO_{\Tilde X_v,x}/m^2_{\Tilde X_v,x}$ of the scheme $\Tilde X_v$.
We know that the $V$-scheme $\Tilde X$ is $V$-smooth at the point $x$.
By Lemma \ref{two_points_and_Cohen}
there are $f_1,...,f_n\in \Gamma(\Tilde X_v,\cO_{\Tilde X_v})$ generating the $k(v)$-algebra $\cO_{\Tilde X_v,x}/m^2_{\Tilde X_v,x}$.
These means that the morphism $\bar f=(\bar f_1,...,\bar f_n): x^{(2)}\to \Aff^n_v$ is a closed embedding.
Then the point $y=\bar f(x)$ is a closed point of $\Aff^n_v$ such that particularly the induced map
$k(v)(x)\leftarrow k(v)(y)$ is a field isomorphism. Using an affine translation of $\Aff^n_v$ by a $k(v)$-rational vector
we may and will suppose below in this proof that $y$ is in $\mathbb G^n_{m,k(v)}$.
\begin{rem}\label{x(2)_and_y(2)}
Since the $\bar f: x^{(2)}\to \Aff^n_v$ is a closed embedding the $\bar f$ identifies $x^{(2)}$
with the closed subscheme $y^{(2)}$ of $\Aff^n_v$.
\end{rem}
Consider projections $q_r: \Aff^n_v\to \Aff^r_v$ (r=1,...,n) and put $y_r=q_r(y)$. Clearly,
$y_r\in \mathbb G^r_{m,k(v)}$ and $y_n=y$.
\begin{rem}\label{h_r_and x_and y_r}
Let $h_1,...,h_n\in \Gamma(\Tilde X_v,\cO_{\Tilde X_v})$ be such that for each $r=1,...,n$ one has
$h_r\equiv f_r\text{mod}\ m^2_{\Tilde X_v,x}$.
Consider morphisms
$h=(h_1,...,h_n): \Tilde X_v\to \Aff^n_v$
and
$h_{(r)}=(h_1,...,h_r): \Tilde X_v\to \Aff^r_v$ ($r=1,...,n$).
Then $h_{(r)}=q_r\circ h$ for $r=1,...,n$ and $h_{(n)}=h$.
Furthermore $h(x)=y$. Thus, for each $r=1,...,n$ the point $h_{(r)}(x)=y_r$ and it belongs to $\mathbb G^r_{m,k(v)}$.
\end{rem}
\begin{rem}\label{z_and_y_r}
Let $z\in \Tilde Z_v$ be a closed point. Suppose $h_r(z)=0$, then $h_{(r)}(z)\notin \mathbb G^r_{m,k(v)}$.
Thus, $h_{(r)}(z)\neq y_r$ and $z\notin h^{-1}_{(r)}(y_r)$.
\end{rem}

\begin{lem}\label{h_r}{\rm
Let $1\leq r\leq n$. Let $h_1,...,h_r\in \Gamma(\Tilde X_v,\cO_{\Tilde X_v})$ be such that for each $i=1,...,r$ one has
$h_i\equiv f_i\text{mod}\ m_{\Tilde X_v,x}$.
Consider morphism
$h_{(r)}=(h_1,...,h_r): \Tilde X_v\to \Aff^r_v$.
Then $h_{(r)}(x)=y_r$.
In particular, $h_{(r)}(x)\in \mathbb G^r_{m,k(v)}$.
}
\end{lem}
\begin{proof}
For each $r<j\leq n$ take an $h_j\in \Gamma(\Tilde X_v,\cO_{\Tilde X_v})$ such that $h_j\equiv f_j\text{mod}\ m_{\Tilde X_v,x}$.
Now by Remark \ref{h_r_and x_and y_r} $h_{(r)}(x)=y_r$ and it belongs to $\mathbb G^r_{m,k(v)}$.
\end{proof}

\begin{notation}
Put $s_0=s_{new}^{\otimes M}\in \Gamma(\bar X,\cM)$, {\it where $M$ and $\cM$ are as above in this proof}.
\end{notation}

\begin{lem}\label{n-1}{\rm
For each $r=1,...,n-1$ there are integers $e_1,...,e_r$ with $e_i$ dividing $e_{i+1}$,
and sections
$s_i\in H^0(\bar X,\cM^{\otimes e_i})$ and functions
$g_i=Res_{\tilde X}(s_i)/Res_{\tilde X}(s^{\otimes e_i}_0)\in \Gamma(\Tilde X,\cO_{\Tilde X})$
with $i=1,...,r$ such that \\
(i) $\text{dim}(\{s_0=s_1=...=s_r=0\}_v)=(n-1)-r$ \\
(ii) for $i=1,...,r$ one has $Res_{x^{(2)}}(g_i)=Res_{x^{(2)}}(f_i)$ in $\Gamma(\tilde X_v,\cO_{\tilde X_v})/m^2_{\Tilde X_v,x}=\Gamma(x^{(2)},\cO_{x^{(2)}})$;\\
(iii) for $p_r=(g_1,...,g_r): \Tilde X\to \Aff^r_V$ the dimension of
$p^{-1}_r(y_r)_v \cap \Tilde Z_v$ is at most $(n-1)-r$;\\
(iv) for the morphism $p_r$ one has $x\in p^{-1}_r(y_r)_v \cap \Tilde Z_v$. \\
}
\end{lem}

\begin{rem}\label{Z'_n-1}
By the items $(iii)$ and $(iv)$ of Lemma \ref{n-1} one has
$p^{-1}_{n-1}(y_{n-1})_v \cap \Tilde Z_v=\{x\}\sqcup Z'$, where either $\text{dim} Z'=0$ or $Z'=\emptyset$.
\end{rem}

\begin{lem}\label{n}{\rm
Take $r=n-1$ and the sections $s_1,...,s_{n-1}$ and the functions $g_1,...,g_{n-1}$ as in Lemma \ref{n-1}.
Then there is
an integer $e_n$ with $e_{n-1}$ dividing $e_{n}$ and
a section $s_n\in H^0(\bar X,\cM^{\otimes e_n})$ and the function
$g_n=Res_{\tilde X}(s_n)/Res_{\tilde X}(s^{\otimes e_n}_0)\in \Gamma(\Tilde X,\cO_{\Tilde X})$
such that \\
(i\ fin) the closed subset $\{s_0=s_1=...=s_{n-1}=s_n=0\}_v$ of $\bar X_v$ is empty;\\
(ii\ fin) for $r=1,...,n$ one has
$Res_{x^{(2)}}(g_r)= Res_{x^{(2)}}(f_r)$ in $\Gamma(\tilde X_v,\cO_{\tilde X_v})/m^2_{\Tilde X_v,x}=\Gamma(x^{(2)},\cO_{x^{(2)}})$;\\
(iii\ fin) for $p=(g_1,...,g_n): \Tilde X\to \Aff^n_V$ one has $p^{-1}(y)_v \cap \Tilde Z_v =\{x\}$;\\
(iv\ fin) $p(x)=y$ and $p|_{x^{(2)}}: x^{(2)}\to y^{(2)}$ is a scheme isomorphism.
}
\end{lem}
Given Lemma \ref{n} complete the proof of Proposition \ref{bar pi}.
To do that consider a $V$-morphism
\begin{equation}
\label{morphism:bar pi_2}
\bar \pi=[s_0:s_1:...:s_n]: \bar X \to \Pro^{n,w}_V,
\end{equation}
where $\Pro^{n,w}_V$ is the weighted projective space of the weight
$(1,e_1,...,e_n)$.
This $V$-morphism is projective and affine. Thus, it is finite. Hence it is surjective.

It is easy to check that $(\Pro^{n,w}_V)_{t_0}\cong \Aff^n_V$.
Since the vanishing locus of the section $s_0$ equals $X_{\infty}$ it follows that
$\bar \pi^{-1}((\Pro^{n,w}_V)_{t_0})=\Tilde X$. This proves the assertion (0).
Write below in this section $\pi$ for $\bar \pi|_{\Tilde X}: \Tilde X\to (\Pro^{n,w}_V)_{t_0}$.

Thus, $\pi: \Tilde X\to \Aff^n_V=(\Pro^{n,w}_V)_{t_0}$ is a finite surjective $V$-morphism between
{\it regular} irreducible schemes.
By the miracle flatness
$\pi: \Tilde X\to \Aff^n_V$ is flat.
The assertion (1) is proved.

Clearly, $\bar \pi^{-1}(\bar \pi(x))\cap \bar Z=\pi^{-1}(\pi(x))\cap \Tilde Z$.
Since
$\pi=(g_1,...,g_n)=p$, we see that $\pi^{-1}(\pi(x))\cap \Tilde Z=\{x\}$.
Thus, $\bar \pi^{-1}(\bar \pi(x))\cap \bar Z=\{x\}$.
The assertion (3) is proved.


By the very construction the morphism $\bar \pi$ enjoys the property
(2). The property (2) yields the property (6).

We already know that the morphism $\pi: \Tilde X\to \Aff^n_V$ is finite and flat.
In particular, the morphism
$\pi:\Tilde X \to \Aff^n_V$
is flat and finite at the point $x\in \Tilde X$.
Now the property (2) yields that $\pi$ is \'{e}tale at the point $x$.
The ramification locus of $\pi$ is a divisor by \cite[Ch.VI, Theorem 6.8]{A-K}.
Thus, the property (4) holds for $\pi$.
The property (5) is clear now.
To prove Proposition \ref{bar pi} it remains to prove Lemmas \ref{n-1} and \ref{n}.

\begin{proof}[Proof of Lemma \ref{n-1}]
Put $p_0: \Tilde X\to V=\Aff^0_V$ and $y_0=v\in V$. Then $p^{-1}_0(y_0)\cap \Tilde Z_v=\Tilde Z_v$.
Since $\tilde X=\bar X-X_{\infty}$ it follows that $X_{\infty}=\bar X-\tilde X$ as sets. Now by
Lemma \ref{X_reg}, the item (3), the dimension of $X_{\infty,v}$ is $n-1$.
Since $\{s_0=0\}=\{s_{new}=0\}=X_{\infty}$ as sets then the dimension of $\{s_0=0\}_v$ is $n-1$.
Thus, the lemma is true for $r=0$ (the exceptional case). The induction step.
Suppose the lemma is true for an integer $r-1\geq 0$ and prove it for the integer $r$.

Consider the closed subset $X_{\infty,v}=\{s_0=0\}_v$ in $\bar X_v$.
Consider the closed subset $Y_{r-1}:=\{s_0=s_1=...=s_{r-1}=0\}_v$ in $\bar X_v$.
Let $Y'_{r-1}$ be a finite closed subset of the $Y_{r-1}$ such that each irreducible component of $Y_{r-1}$
contains at least one point $y'$ of $Y'_{r-1}$. Since $Y'_{r-1}\subseteq X_{\infty,v}$ one has $x\notin Y'_{r-1}$.

Put $Z_{r-1}=p^{-1}_{r-1}(y_{r-1})\cap \Tilde Z_v$.
Choose a finite subset $Z'_{r-1}$ of closed points in $Z_{r-1}$ such that
each irreducible component of $Z_{r-1}$
contains at least one point $z'$ of $Z'$ and moreover $x\notin Z'_{r-1}$.
Since $\Tilde Z_v\subset \Tilde X_v$ one has $Z'_{r-1}\cap Y'_{r-1}=\emptyset$.


So, we get a closed finite subscheme $Y''_{r-1}:=(Z'_{r-1}\sqcup Y'_{r-1})\sqcup x^{(2)}$ in $\bar X_v$.
By the Claim
there is an integer $e_r$ with $e_{r-1}$ dividing $e_r$ and a section
$s_r\in H^0(\bar X,\cM^{\otimes e_r})$ such that \\
(i) $s_r$ does not vanish at any of points $y'$'s ($y'\in Y'_{r-1})$;\\
(ii) $s_r(z')=0$ for any of points $z'$'s ($z'\in Z'_{r-1})$; \\
(iii) $Res_{x^{(2)}}(s_r)=f_r\cdot Res_{x^{(2)}}(s^{\otimes e_r}_0)$ in $\Gamma(x^{(2)},\cM^{\otimes e_r}|_{x^{(2)}})$. \\\\
Now the closed subset $Y_r:=\{s_0=s_1=...=s_r=0\}_v$ of $\bar X_v$ has dimension $(n-1)-r$;\\
Put $g_r=Res_{\tilde X}(s_r)/Res_{\tilde X}(s^{\otimes e_r}_0)\in \Gamma(\Tilde X,\cO_{\Tilde X})$
and take
$p_r=(g_1,...,g_r): \Tilde X\to \Aff^r_V$; then
$\text{dim}(p^{-1}_r(y_r)_v \cap \Tilde Z_v)\leq (n-1)-r$;\\
(indeed, by Remark \ref{z_and_y_r} $z'\notin p^{-1}_r(y_r)_v$ for any $z'\in Z'_{r-1}$);  \\
$Res_{x^{(2)}}(g_r)=Res_{x^{(2)}}(f_r)$ in $\Gamma(x^{(2)},\cO_{x^{(2)}})=\Gamma(\tilde X_v,\cO_{\tilde X_v})/m^2_{\Tilde X_v,x}$;\\
$x\in p^{-1}_r(y_r)_v \cap \Tilde Z_v$ by Lemma \ref{h_r}.

The proof of the lemma is completed.
\end{proof}

\begin{proof}[Proof of Lemma \ref{n}]
By Remark \ref{Z'_n-1} one has
$p^{-1}_{n-1}(y_{n-1})_v \cap \Tilde Z_v=\{x\}\sqcup Z'$, where either $\text{dim} Z'=0$ or $Z'=\emptyset$.
Consider the closed subset $Y_{n-1}:=\{s_0=s_1=...=s_{n-1}=0\}_v$ in $\bar X_v$.
Recall that
$\{s_0=0\}=X_{\infty}$ as sets and $\tilde X=\bar X-X_{\infty}$.
Since $\Tilde Z_v\subset \Tilde X_v$,  and  one has $\Tilde Z_v\cap Y_{n-1}=\emptyset$.
By Lemma \ref{n-1} $\text{dim}(Y_{n-1})=0$.
So, we have a closed finite subscheme $Y'':=(Z'\sqcup Y_{n-1})\sqcup x^{(2)}$ in $\bar X_v$.

By the Claim there is an integer $e_n$ with $e_{n-1}$ dividing $e_{n}$ and
a section $s_n\in H^0(\bar X,\cM^{\otimes e_n})$ such that \\
(i) $s_n$ does not vanish at any of points $y$'s ($y\in Y_{n-1})$;\\
(ii) $s_n(z')=0$ for any of points $z'$'s ($z'\in Z')$; \\
(iii) $Res_{x^{(2)}}(s_n)= f_n\cdot Res_{x^{(2)}}(s^{\otimes e_n}_0)$ in $\Gamma(x^{(2)},\cM^{\otimes e_n}|_{x^{(2)}})$.\\\\
Put $g_n=Res_{\tilde X}(s_n)/Res_{\tilde X}(s^{\otimes e_n}_0)\in \Gamma(\Tilde X,\cO_{\Tilde X})$.
Now the following holds \\
(i\ fin) the closed subset $\{s_0=s_1=...=s_{n-1}=s_n=0\}_v$ of $\bar X_v$ is empty;\\
(ii\ fin) for $r=1,...,n$ one has
$Res_{x^{(2)}}(g_r)= Res_{x^{(2)}}(f_r)$ in $\Gamma(\tilde X_v,\cO_{\tilde X_v})/m^2_{\Tilde X_v,x}=\Gamma(x^{(2)},\cO_{x^{(2)}})$;\\
(iii\ fin) for $p=(g_1,...,g_n): \Tilde X\to \Aff^n$ by Remark \ref{z_and_y_r} one has $p^{-1}(y)_v \cap \Tilde Z_v =\{x\}$;\\
(iv\ fin) by Lemma \ref{h_r} $p(x)=y$ and by Lemma \ref{x(2)_and_y(2)} the morphism
$p|_{x^{(2)}}: x^{(2)}\to y^{(2)}$ is a scheme isomorphism.

The proof of the lemma is completed.
\end{proof}
The proof of Proposition \ref{bar pi} is completed.
\end{proof}

{\it Step 4}.
\begin{notation}{\rm
Following Lemma \ref{Just_Etale_at_x} and Proposition \ref{bar pi} $x\in X^{\circ}_v$. Put \\
$Z^{\circ}=\Tilde Z\cap X^{\circ}$,
${\cal W}=Spec\ \cO_{\Aff^n_V,y}$,
$\cU:=\Tilde X\times_{\Aff^n_V} \cal W$,\\
${\cal E}:= E\times_{\Aff^n_V} \cal W$,
$\cU^{\circ}=X^{\circ}\times_{\Aff^n_V} \cal W = \cU-{\cal E}$, \\
${\cal Z}:=\Tilde Z\times_{\Aff^n_V} \cal W$,\
${\cal Z}^{\circ}:=Z^{\circ}\times_{\Aff^n_V} \cal W$, \\
$\pi_{\cal W}=\pi\times_{\Aff^n_V} \cal W: \cU \to \cal W$,\\
$\pi^{\circ}_{\cal W}=\pi_{\cal W}|_{\cU^{\circ}}: \cU^{\circ} \to \cal W$. \\
Note that $x$ is a {\it closed point} of $\cU^{\circ}_v$.\\
Write shortly
$\pi$ for $\pi_{\cal W}$ and $\pi^{\circ}$ for $\pi^{\circ}_{\cal W}$
till the end of this Section.
}
\end{notation}
{\it Step 5}.
\begin{proof}[The end of proof of Theorem \ref{geometric_form_of_O-G-L}]
Since $\pi|_{X^{\circ}}$ is \'{e}tale for the affine open
$X^{\circ}\subset \Tilde X$ containing $x$ it follows that the morphism
$\pi^{\circ}: \cU^{\circ}\to \cal W$ is \'{e}tale.
Recall that $y:=\pi(x)$ is the closed point of $\cal W$.
Since an \'{e}tale morphism is open it follows that
$\pi^{\circ}: \cU^{\circ}\to \cal W$
is surjective.

The property (3) of the morphism $\pi$ yields that the closed point $x$ is a unique closed point of
the closed subset ${\cal Z}$ of the semi-local scheme $\cU$.
Hence the closed subset ${\cal Z}$ is local.
Thus, ${\cal Z}\subset \cU^{\circ}$. So,
${\cal Z}^{\circ}={\cal Z}$, ${\cal Z}^{\circ}$ is local and
the point $x\in {\cal Z}^{\circ}$ {\it its unique closed point}.

Put $\cal S:=\pi(\cal Z^{\circ})=\pi(\cal Z)\subset \cal W$. Then the subset $\cal S$ is closed in $\cal W$, since
the morphism $\pi$ is finite. Now the properties (1), (2), (3) and (4), and the Nakayama lemma yield that
the morphism of reduced schemes
$\pi|_{\cal Z}: {\cal Z}\to \cal S$
is a scheme isomorphism. Thus, the morphism
$\pi|_{{\cal Z}^{\circ}}: {\cal Z}^{\circ}\to \cal S$ is a reduced scheme isomorphism.

Since $\Tilde Z\subset \Tilde X$ is a closed subset of {\it pure codimension one} in $\Tilde X$
it follows that ${\cal Z^{\circ}}$ is closed of pure codimension one in $\cU^{\circ}$.

The morphism $\pi^{\circ}:  \cU^{\circ}\to \cal W$ is \'{e}tale. Hence its base change
$\pi^{\circ}|_{(\pi^{\circ})^{-1}(\cal S)}: (\pi^{\circ})^{-1}(\cal S)\to \cal S$
is also \'{e}tale. The inverse isomorphism
$s: \cal S\to \cal Z^{\circ}$ to the isomorphism
$\pi|_{\cal Z^{\circ}}: \cal Z^{\circ}\to \cal S$
is a section of the \'{e}tale morphism
$(\pi^{\circ})^{-1}(\cal S)\to {\cal S}$.
Thus, the section
$s: {\cal S}\to (\pi^{\circ})^{-1}(\cal S)$
is also \'{e}tale. So, its image is open in $(\pi^{\circ})^{-1}(\cal S)$.
At the same time it is closed as the image of a section.
Thus, $(\pi^{\circ})^{-1}(\cal S)={\cal Z}^{\circ} \sqcup \cal Z'$ (the disjoint union of closed subsets).
We already know that $x\in {\cal Z}^{\circ}$. Thus, $x\notin \cal Z'$.

The scheme ${\cal U}$ is affine regular and semi-local. Thus, it is regular and factorial. So,
there is an $f\in \Gamma({\cal U}, \cO_{\cal U})$ such at $\{f=0\}={\cal E}$.
Hence ${\cal U^{\circ}}={\cal U}_f$ is a principal open subset in ${\cal U}$.
Thus, ${\cal U^{\circ}}={\cal U}_f$ is also regular and factorial.
Thus, there is an $h\in \Gamma({\cal U^{\circ}}, \cO_{\cal U^{\circ}})$
such that $\{h=0\}={\cal Z^{\circ}}$ in ${\cal U}^{\circ}$
and the following square with
$\tau=\pi|_{{\cal U}^{\circ}_h}$
and
$\tau^{-}=\pi|_{{\cal U}^{\circ}_h-{\cal Z}^{\circ}}$
\begin{equation}
\label{Big_Diagram_3}
    \xymatrix{
{\cal U}^{\circ}_h - {\cal Z}^{\circ} \ar[d]^{\tau^{-}}  \ar[rr]^{In} && {\cal U}^{\circ}_h \ar[d]^{\tau} & {\cal Z}^{\circ} \ar[d]^{\cong} &\\
\cal W-\cal S  \ar[rr]^{in} &&   \cal W  &  \cal S  &,\\    }
\end{equation}
is an elementary distinguished square.

The property
$\tau(x)=y$ was already mentioned above in this proof.
It remains to find $g\in m_y$ and to check that
$\{\tau^*(g)=0\}={\cal Z}^{\circ}$ as the Cartier divisors of
the ${\cal U}^{\circ}_h$.
The affine scheme ${\cal W}$ is regular and local. Thus, it is regular and factorial.
Hence there is an $g\in m_y$ such that
$\{g=0\}={\cal S}$ in ${\cal W}$.
Clearly, $\{\tau^*(g)=0\}={\cal Z}^{\circ}$ in ${\cal U}^{\circ}_h$.

Since $\pi|_{X^{\circ}}$ is \'{e}tale for the affine open
$X^{\circ}\subset \Tilde X$ containing $x$ it follows that
the scheme $X^{\circ}$ is $V$-smooth.
The scheme ${\cal U}^{\circ}_h$ is
a {\it localization} of the $V$-smooth affine scheme $X^{\circ}$ (see Definition \ref{principal_open}).
This complete the proof of
Theorem \ref{geometric_form_of_O-G-L}.

\end{proof}
\end{proof}

\end{document}